\documentclass[a4paper, 12pt]{amsart}

\usepackage{amsmath}
\usepackage{amssymb}
\usepackage{amsthm}
\usepackage[mathscr]{eucal}
\newcommand{\ABH}{\operatorname{ABH}}
\newcommand{\ANR}{\operatorname{ANR}}
\newcommand{\cat}{\operatorname{cat}}
\newcommand{\CAT}{\operatorname{CAT}}
\newcommand{\GCBA}{\operatorname{GCBA}}
\newcommand{\LGC}{\operatorname{LGC}}

\newcommand{\id}{\operatorname{id}}
\newcommand{\R}{\operatorname{\mathbb{R}}}
\newcommand{\Z}{\operatorname{\mathbb{Z}}}
\newcommand{\N}{\operatorname{\mathbb{N}}}
\newcommand{\Sph}{\operatorname{\mathbb{S}}}

\newcommand{\Haus}{\operatorname{\mathcal{H}}}
\numberwithin{equation}{section}
\theoremstyle{plain}
\newtheorem{thm}{Theorem}[section]
\newtheorem{lem}[thm]{Lemma}  
\newtheorem{prop}[thm]{Proposition}

\theoremstyle{definition}

\newtheorem{exmp}{Example}[section]
\newtheorem{prob}{Problem}[section]

\theoremstyle{remark}
\newtheorem{rem}{Remark}[section]

\begin{document} 

\title
[Volume pinching theorems for CAT(1) spaces]
{Volume pinching theorems for \\ CAT(1) spaces}

\author
[Koichi Nagano]{Koichi Nagano}

\thanks{
Partially supported
by JSPS KAKENHI Grant Numbers 26610012, 21740036, 18740023,
and by the 2004 JSPS Postdoctoral Fellowships for Research Abroad.}

\address
[Koichi Nagano]
{\endgraf 
Institute of Mathematics, University of Tsukuba
\endgraf
Tennodai 1-1-1, Tsukuba, Ibaraki, 305-8571, Japan}

\email{nagano@math.tsukuba.ac.jp}

\date{January 15, 2021}


\begin{abstract}
We examine volume pinching problems of $\CAT(1)$ spaces.
We characterize a class of compact geodesically complete $\CAT(1)$ spaces
of small specific volume.
We prove a sphere theorem 
for compact $\CAT(1)$ homology manifolds of small volume.
We also formulate a criterion of manifold recognition for homology manifolds
on volume growths under an upper curvature bound.
\end{abstract}

\maketitle

%%%%%%%%%%
%%%%%%%%%%

.%%%%%%%%%%
%Section 1: Introduction
%%%%%%%%%%
\section{Introduction}

%%%%%%%%%%
\subsection{Backgrounds}

Many problems of pinching theorems, including sphere theorems, 
on various metric invariants
have attracted our interests in global Riemannian geometry.
In this paper,
we examine volume pinching problems of $\CAT(1)$ spaces
as a subsequent study of the series of the works of Lytchak and the author 
\cite{lytchak-nagano1} and \cite{lytchak-nagano2}.

For every metric space with an upper curvature bound,
all the spaces of directions are $\CAT(1)$.
Lytchak and the author have proved 
a \emph{local topological regularity theorem}
\cite[Theorem 1.1]{lytchak-nagano2}:
A locally compact metric space with an upper curvature bound
is a topological $n$-manifold
if and only if
all the spaces of directions are homotopy equivalent to an $(n-1)$-sphere.
Once we would establish a sphere theorem for $\CAT(1)$ spaces,
we could obtain an infinitesimal characterization of topological manifolds
for spaces with an upper curvature bound.

We say that a triple of points in a $\CAT(1)$ space is a 
\emph{tripod}
if the three points have pairwise distance at least $\pi$.
Lytchak and the author have invented 
a \emph{capacity sphere theorem} 
\cite[Theorem 1.5]{lytchak-nagano2} 
for $\CAT(1)$ spaces:
If a compact, geodesically complete $\CAT(1)$ space admits no tripod,
then it is homeomorphic to a sphere.

Throughout this paper,
we denote by $\Sph^n$ the standard unit $n$-sphere,
and by $T$ the discrete metric space 
consisting of three points with pairwise distance $\pi$.
For instance,
the spherical join
$\Sph^{n-1} \, \ast \, T$ is a compact, geodesically complete
$\CAT(1)$ space containing the tripod $T$.

We say that a separable metric space is 
\emph{purely $n$-dimensional}
if every non-empty open subset has finite (Lebesgue) covering dimension $n$.
We denote by $\Haus^n$ the $n$-dimensional Hausdorff measure.
If $X$ is a purely $n$-dimensional, compact, 
geodesically complete $\CAT(1)$ space,
then $\Haus^n(X) \ge \Haus^n(\Sph^n)$;
the equality holds if and only if
$X$ is isometric to $\Sph^n$ (\cite[Lemma 3.1 and Proposition 7.1]{nagano1});
moreover, if $\Haus^n(X)$ is sufficiently close to $\Haus^n(\Sph^n)$,
then $X$ is bi-Lipschitz homeomorphic to $\Sph^n$ (\cite[Theorem 1.10]{nagano1}).
Lytchak and the author have proved 
a \emph{volume sphere theorem} 
\cite[Theorem 8.3]{lytchak-nagano2} 
for $\CAT(1)$ spaces:
If a purely $n$-dimensional, compact, geodesically complete
$\CAT(1)$ space $X$ satisfies 
\begin{equation}
\Haus^n(X) < \frac{3}{2} \Haus^n(\Sph^n),
\tag{$\ast$}
\label{eqn: volsphthma}
\end{equation}
then $X$ is homeomorphic to $\Sph^n$.

In the volume sphere theorem \cite[Theorem 8.3]{lytchak-nagano2}, 
the pureness on the dimension is essential 
since we can construct counterexamples
possessing lower dimensional subsets.
The assumption \eqref{eqn: volsphthma} of $\Haus^n$ is optimal
since the spherical join $\Sph^{n-1} \, \ast \, T$ satisfies
$\Haus^n(\Sph^{n-1} \, \ast \, T) = (3/2) \Haus^n(\Sph^n)$.

%%%%%%%%%%
\subsection{Main results}

We construct a $\CAT(1)$ $n$-sphere admitting a tripod
whose $n$-dimensional Hausdorff measure is equal to 
$(3/2) \Haus^n(\Sph^n)$.

\begin{exmp}\label{exmp: triplex}
The spherical join $\Sph^{n-2} \ast T$
can be represented by the quotient metric space
$\bigsqcup_{i=1,2,3} \Sph_{+, i}^{n-1} / \sim$
obtained by gluing three closed unit $(n-1)$-hemispheres $\Sph_{+, i}^{n-1}$
along their boundaries 
$\partial \Sph_{+, i}^{n-1} = \partial \Sph_{+, j}^{n-1}$.
For $i = 1, 2, 3, 3+1=1$, 
let $\Sigma_i^{n-1}$ be
the isometrically embedded unit $(n-1)$-spheres 
$\Sph_{+, i}^{n-1} \sqcup \Sph_{+, i+1}^{n-1} / \sim$ in $\Sph^{n-2} \ast T$
obtained by the relation
$\partial \Sph_{+, i}^{n-1} = \partial \Sph_{+, i+1}^{n-1}$.
We take three copies of closed unit $n$-hemispheres 
$\Sph_{+, i}^n$, $i = 1, 2, 3$.
Let $X$ be the quotient metric space obtained as
\[
X := (\Sph^{n-2} \ast T) \sqcup \bigl( \bigsqcup_{i=1,2,3} \Sph_{+, i}^n \bigr) / \sim
\]
by attaching $\Sph_{+, i}^n$ to $\Sph^{n-2} \ast T$
along $\Sigma_i^{n-1} = \partial \Sph_{+, i}^n$ for each $i \in \{ 1, 2, 3 \}$.
We call $X$ the 
\emph{$n$-triplex}.
The $n$-triplex $X$ is a purely $n$-dimensional, compact, geodesically complete
$\CAT(1)$ space that is homeomorphic to $\Sph^n$.
This space has a tripod and satisfies
$\Haus^n(X) = (3/2) \Haus^n(\Sph^n)$.
We notice that the $1$-triplex is by definition a circle of length $3\pi$.
\end{exmp}

As one of the main results,
we prove the following characterization:

\begin{thm}\label{thm: critical}
Let $X$ be a purely $n$-dimensional, compact, 
geodesically complete $\CAT(1)$ space.
If $X$ satisfies 
\begin{equation}
\Haus^n(X) = \frac{3}{2} \Haus^n(\Sph^n),
\label{eqn: criticala}
\end{equation}
then $X$ is either
homeomorphic to $\Sph^n$ or
isometric to $\Sph^{n-1} \ast T$.
If in addition $X$ has a tripod,
then $X$ is isometric to either the $n$-triplex
or $\Sph^{n-1} \ast T$.
\end{thm}

Theorem \ref{thm: critical} for the case of $n \le 2$ 
was proved in \cite{nagano2}.

For $\CAT(1)$ homology manifolds,
one can hope that Theorem \ref{thm: critical}
enables us to relax the condition \eqref{eqn: volsphthma}
in the volume sphere theorem \cite[Theorem 8.3]{lytchak-nagano2}.
We note that
every $\CAT(1)$ homology manifold (without boundary) 
is geodesically complete.

The other main result is the following volume sphere theorem
for $\CAT(1)$ homology manifolds:

\begin{thm}\label{thm: volsphhm}
For every positive integer $n$,
there exists a sufficiently small positive number 
$\delta \in (0,\infty)$ depending only on $n$
such that
if a compact $\CAT(1)$ homology $n$-manifold $X$ satisfies
\begin{equation}
\Haus^n(X) < \frac{3}{2} \Haus^n(\Sph^n) + \delta,
\label{eqn: volsphhma}
\end{equation}
then $X$ is homeomorphic to $\Sph^n$.
\end{thm}

Theorem \ref{thm: volsphhm} is new even for Riemannian manifolds.
We notice that 
a complete Riemannian manifold is $\CAT(1)$ 
if and only if 
it has sectional curvature $\le 1$ and injectivity radius $\ge \pi$.

\begin{rem}\label{rem: aftervolsphhm}
Let $M$ be a simply connected, compact,
$(2n)$-dimensional Riemannian manifold 
of positive sectional curvature $\le 1$.
Due to the Klingenberg estimate of injectivity radii,
we see that $M$ has injectivity radius $\ge \pi$;
in particular, $M$ is $\CAT(1)$.
By the sphere theorem of Coghlan and Itokawa \cite{coghlan-itokawa},
we know that 
if $\Haus^{2n}(M) \le (3/2) \Haus^{2n}(\Sph^{2n})$,
then $M$ is homeomorphic to $\Sph^{2n}$.
In the sphere theorem of Coghlan and Itokawa \cite{coghlan-itokawa},
the condition on the volume was relaxed by Wu \cite{jywu},
and by Wen \cite{ywen1, ywen2}
under lower sectional curvature bounds.
In the proofs in \cite{jywu} and in \cite{ywen1, ywen2},
the assumptions of the lower sectional curvature bounds 
for Riemannian manifolds is essential.
\end{rem}

In the proof of Theorem \ref{thm: volsphhm}, 
we need a criterion of manifold recognition for homology manifolds
on volume growths.
For $\kappa \in \R$,
we denote by $M_{\kappa}^n$
the simply connected, complete Riemannian $n$-manifold
of constant curvature $\kappa$. 
Let $D_{\kappa}$ denote
the diameter of $M_{\kappa}^n$.
For $r \in (0,D_{\kappa})$,
we denote by $\omega_{\kappa}^n(r)$ 
the $n$-dimensional Hausdorff measure of any metric ball in $M_{\kappa}^n$
of radius $r$ if $n \ge 2$,
and by $\omega_{\kappa}^1(r)$ 
the $1$-dimensional Hausdorff measure of $[-r,r]$.
From the local topological regularity theorem \cite[Theorem 1.1]{lytchak-nagano2}
and the volume sphere theorem \cite[Theorem 8.3]{lytchak-nagano2},
we deduce a
\emph{local topological regularity theorem on volume growths}:
Let $X$ be a locally compact, geodesically complete $\CAT(\kappa)$ space,
and let $W$ be a purely $n$-dimensional open subset of $X$.
If for every $x \in W$ there exists $r \in (0,D_{\kappa})$ satisfying
$\Haus^n \bigl( B_r(x) \bigr) / \omega_{\kappa}^n(r) < 3/2$,
then $W$ is a topological $n$-manifold,
where $B_r(x)$ is the closed metric ball of radius $r$ centered at $x$
(see Theorem \ref{thm: vgtopreg}).

As one of the key ingredients in the proof of Theorem \ref{thm: volsphhm}, 
we provide the following criterion of manifold recognition:

\begin{thm}\label{thm: vgtopreghm}
For every positive integer $n$,
there exists a sufficiently small
positive number $\delta \in (0,\infty)$ depending only on $n$
with the following property:
Let $X$ be a $\CAT(\kappa)$ homology $n$-manifold,
and let $W$ be an open subset of $X$.
If for every $x \in W$ there exists $r \in (0,D_{\kappa})$ satisfying
\begin{equation}
\frac{\Haus^n \bigl( B_r(x) \bigr)}{\omega_{\kappa}^n(r)} < \frac{3}{2} + \delta,
\label{eqn: vgtopreghma}
\end{equation}
then $W$ is a topological $n$-manifold.
\end{thm}

%%%%%%%%%%
\subsection{Outline}

The organization of this paper is as follows:
In Section 2,
we recall the known basic properties
of metric spaces with an upper curvature bound.
In Section 3,
we deduce the local topological regularity theorem 
(Theorem \ref{thm: vgtopreg})
on volume growths
mentioned above.

In Section 4,
we prove Theorem \ref{thm: critical}.
Due to the capacity sphere theorem for $\CAT(1)$ spaces
\cite{lytchak-nagano2},
it suffices to consider the case where  
$X$ is a purely $n$-dimensional, compact, geodesically complete $\CAT(1)$ space 
with \eqref{eqn: criticala} admitting a tripod.
By the volume rigidity of Bishop-G\"{u}nther type \cite{nagano1},
the space $X$ consists of three unit $n$-hemispheres.
Observing how the hemispheres meet each other,
we obtain the conclusion.
When we determine the geometric structure,
we use the volume sphere theorem for $\CAT(1)$ spaces
\cite{lytchak-nagano2}.

In Section 5,
we prove Theorems \ref{thm: volsphhm} and \ref{thm: vgtopreghm}
by contradiction.
To achieve the tasks,
we use Theorem \ref{thm: critical},
the local topological regularity theorem \cite{lytchak-nagano2},
and the volume convergence theorem \cite{nagano1}.

%%%%%%%%%%
\subsection{Problem}

As a natural question beyond Theorem \ref{thm: volsphhm}, 
we pose the following volume pinching problem 
for $\CAT(1)$ spaces:

\begin{prob}\label{prob: pinching-problem}
Let $n \ge 2$.
Let $R_n$ be the supremum of $R \in (3/2,\infty)$
for which every $\CAT(1)$ homology $n$-manifold $X$
with $\Haus^n(X) / \Haus^n(\Sph^n) \le R$
is homeomorphic to $\Sph^n$.
\begin{enumerate}
\item
Find the concrete value $R_n$.
\item
Describe all compact $\CAT(1)$ homology $n$-manifolds $X$
satisfying $\Haus^n(X) / \Haus^n(\Sph^n) = R_n$
in the maximal critical case.
\end{enumerate}
\end{prob}

This problem seems to be interesting even for 
Riemannian manifolds.

A Riemannian manifold $M$ is said to be a 
\emph{$C_l$-manifold}
if every geodesic in $M$ is contained in a periodic closed geodesic of length $l$;
in this case,
the Riemannian metric of $M$ is called a 
\emph{$C_l$-metric}.
For any $n$-dimensional $C_{2\pi}$-manifold $M$,
the volume ratio $\Haus^n(M) / \Haus^n(\Sph^n)$
is an integer (\cite[Theorem A]{weinstein}),
called the \emph{Weinstein integer for $M$}
(see \cite[Theorem 2.21]{besse}).
We know the concrete values of the Weinstein integers for 
compact symmetric spaces of rank one
with the standard $C_{2\pi}$-metric
(see \cite[VI.7]{berger} and \cite[2.23]{besse}).

Every compact symmetric space of rank one with the standard $C_{2\pi}$-metric
is $\CAT(1)$.
The number $R_n$ in Problem \ref{prob: pinching-problem}
is not greater than the Weinstein integers
for the projective spaces with the $C_{2\pi}$-metric.

%%%%%%%%%%
\subsection*{Acknowledgments}

The author would like to express his gratitude to
Alexander Lytchak for many valuable discussions and comments.
The author would like to thank an anonymous referee
for providing the proof of Proposition \ref{prop: building}
much shorter than a proof in a previous manuscript of this paper
(see Remark \ref{rem: referee}).

%%%%%%%%%%
%%%%%%%%%%

%%%%%%%%%%
%Section 2: Preliminaries
%%%%%%%%%%
\section{Preliminaries}

We refer the readers to 
\cite{alexander-kapovitch-petrunin}, \cite{alexandrov-berestovskii-nikolaev},
\cite{ballmann}, \cite{bridson-haefliger}, \cite{burago-burago-ivanov},
\cite{buyalo-schroeder} 
for the basic facts
on metric spaces with an upper curvature bound.

%%%%%%%%%%%
\subsection{Metric spaces}

We denote by $d$ the metrics on metric spaces.
For $r \in (0,\infty)$, and for a point $p$ in a metric space,
we denote by $U_r(p)$, $B_r(p)$, and $\partial B_r(p)$ 
the open metric ball of radius $r$ centered at $p$,
the closed one, and the metric sphere, respectively.
A metric space is said to be 
\emph{proper}
if every closed metric ball is compact.

For a metric space $X$,
we denote by $C(X)$ the Euclidean cone over $X$.
For metric spaces $Y$ and $Z$,
we denote by $Y \ast Z$ the spherical join of $Y$ and $Z$.
Note that $C(Y \ast Z)$ is isometric to the $\ell^2$-direct product $C(Y) \times C(Z)$.
The spherical join $\Sph^{m-1} \ast \Sph^{n-1}$ is isometric to $\Sph^{m+n-1}$.

For $r \in (0,\infty]$,
a metric space $X$ is said to be
\emph{$r$-geodesic}
if every pair of points $p, q$ with distance $< r$
can be joined by a geodesic $pq$ in $X$,
where a \emph{geodesic} $pq$ means 
the image of an isometric embedding $\gamma \colon [a,b] \to X$
from a closed interval $[a,b]$
with $\gamma(a) = p$ and $\gamma(b) = q$.
A metric space is 
\emph{geodesic}
if it is $\infty$-geodesic.
A geodesic space is proper
if and only if
it is complete and locally compact.

For $r \in (0,\infty]$,
a subset $C$ in a metric space is said to be 
\emph{$r$-convex}
if $C$ itself is $r$-geodesic as a metric subspace, 
and if every geodesic joining two points in $C$
is contained in $C$.
A subset $C$ in a metric space is 
\emph{convex}
if $C$ is $\infty$-convex.

%%%%%%%%%%%
\subsection{CAT$\boldsymbol{(\kappa)}$ spaces}

For $\kappa \in \R$,
a complete metric space $X$ is said to be 
$\CAT(\kappa)$
if $X$ is $D_{\kappa}$-geodesic,
and if every geodesic triangle in $X$ with perimeter $< 2D_{\kappa}$
is not thicker than the comparison triangle in $M_{\kappa}^2$.
Our $\CAT(\kappa)$ spaces are assumed to be complete.
A metric space 
\emph{has an upper curvature bound $\kappa$}
if every point has a $\CAT(\kappa)$ neighborhood.

Let $X$ be a $\CAT(\kappa)$ space.
Every pair of points in $X$ with distance $< D_{\kappa}$
can be uniquely joined by a geodesic.
Let $p \in X$ be arbitrary.
For every $r \in (0,D_{\kappa}/2]$,
the ball $B_r(p)$ is convex.
Along the geodesics emanating from $p$,
for every $r \in (0,D_{\kappa})$
the ball $B_r(p)$ is contractible inside itself.
Every open subset of $X$ is an $\ANR$ (absolute neighborhood retract)
(\cite{ontaneda}, \cite{kramer}).
For $x, y \in U_{D_{\kappa}}(p) - \{p\}$,
we denote by $\angle_p(x,y)$ the angle at $p$
between $px$ and $py$.
Put $\Sigma_p'X := \{ \, px \mid x \in U_{D_{\kappa}}(p) - \{p\} \, \}$.
The angle $\angle_p$ at $p$
is a pseudo-metric on $\Sigma_p'X$.
The 
\emph{space of directions $\Sigma_pX$ at $p$}
is defined as the $\angle_p$-completion of
the quotient metric space $\Sigma_p'X / \angle_p = 0$.
For $x \in U_{D_{\kappa}}(p) - \{p\}$,
we denote by $x_p' \in \Sigma_pX$
the starting direction of $px$ at $p$.
The 
\emph{tangent space $T_pX$ at $p$} 
is defined as the Euclidean cone $C(\Sigma_pX)$
over $\Sigma_pX$.
We denote by $o_p \in T_pX$ the vertex of the cone $T_pX$.
The space $\Sigma_pX$ is $\CAT(1)$,
and the space $T_pX$ is $\CAT(0)$.
In fact,
for a metric space $\Sigma$,
the Eulcidean cone $C(\Sigma)$ is $\CAT(0)$ if and only if $\Sigma$ is $\CAT(1)$.
For metric spaces $Y$ and $Z$,
the spherical join $Y \ast Z$ is $\CAT(1)$
if and only if $Y$ and $Z$ are $\CAT(1)$.

%%%%%%%%%%%
\subsection{Geodesically complete CAT$\boldsymbol{(\kappa)}$ spaces}

We refer the readers to \cite{lytchak-nagano1} 
for the basic properties  of $\GCBA$ spaces, that is,
locally compact, separable, locally geodesically complete 
spaces with an upper curvature bound.
Recall that
a $\CAT(\kappa)$ space is said to be 
\emph{locally geodesically complete}
(or \emph{has geodesic extension property})
if every geodesic defined on a compact interval can be extended to
a local geodesic beyond endpoints.
A $\CAT(\kappa)$ space is said to be 
\emph{geodesically complete}
if every geodesic can be extended to
a local geodesic defined on $\R$.
Every locally geodesically complete $\CAT(\kappa)$ space
is geodesically complete.
The geodesic completeness for compact (resp.~proper) $\CAT(\kappa)$ spaces
is preserved under the (resp.~pointed) Gromov-Hausdorff limit.

Let $X$ be a proper, geodesically complete $\CAT(\kappa)$ space.
For every $p \in X$,
the space $\Sigma_pX$ coincides with the set $\Sigma_p'X$
of all starting directions at $p$.
Moreover, 
$\Sigma_pX$ is compact and geodesically complete,
and $T_pX$ is proper and geodesically complete.
In fact, for a $\CAT(1)$ space $\Sigma$,
the cone $C(\Sigma)$ is geodesically complete 
if and only if $\Sigma$ is geodesically complete and not a singleton.
For $\CAT(1)$ spaces $Y$ and $Z$,
the join $Y \ast Z$ is geodesically complete
if and only if so are $Y$ and $Z$.

%%%%%%%%%%%
\subsection{Dimension of CAT$\boldsymbol{(\kappa)}$ spaces}

Let $X$ be a separable $\CAT(\kappa)$ space.
The (Lebesgue) covering dimension $\dim X$ satisfies
\[
\dim X = 1 + \sup_{p \in X} \dim \Sigma_pX = \sup_{p \in X} \dim T_pX
\]
(\cite{kleiner}).
Let $X$ be proper and geodesically complete.
Every relatively compact open subset of $X$
has finite covering dimension
(see \cite[Subsection 5.3]{lytchak-nagano1}).
The dimension $\dim X$ 
is equal to the Hausdorff dimension of $X$;
moreover, $\dim X$ is equal to the supremum of $m$
such that $X$ admits an open subset $U$ homeomorphic to 
the Euclidean $m$-space $\R^m$
(\cite[Theorem 1.1]{lytchak-nagano1}).
If $\dim X = n$,
then the support of $\Haus^n$ coincides with
the set of all points $x \in X$ with $\dim \Sigma_xX = n-1$
(\cite[Theorem 1.2]{lytchak-nagano1}). 

From the studies in \cite[Subsection 11.3]{lytchak-nagano1}
on the stability of dimension,
we can immediately derive the following three lemmas:

\begin{lem}\label{lem: lstab}
Let $(X_i,p_i)$, $i = 1, 2, \dots$, be a sequence
of pointed proper geodesically complete $\CAT(\kappa)$ spaces.
Assume that
$(X_i,p_i)$ converges to some $(X,p)$
in the pointed Gromov-Hausdorff topology.
Then 
\[
\dim X \le
\liminf_{i \to \infty} \dim X_i.
\]
\end{lem}

\begin{proof}
Assume that for some positive integer $n$
there exists $x_n \in X$ with $\dim \Sigma_{x_n}X = n-1$. 
In this case, we have $\dim X \ge n$.
We can take a sequence 
$x_{n,i} \in X_i$, $i = 1, 2, \dots$, converging to the point $x_n \in X$.
Since $\dim \Sigma_{x_n}X = n-1$,
there exists $r_n \in (0,D_{\kappa})$ such that
we have $\dim U_{r_n}(x_{n,i}) = n$ for all sufficiently large $i$
(\cite[Lemma 11.5]{lytchak-nagano1}).
This implies $n \le \liminf_{i \to \infty} \dim X_i$,
and the lower semi-continuity.
\end{proof}

On the Gromov-Hausdorff topology, we have:

\begin{lem}\label{lem: stab}
Let $X_i$, $i = 1, 2, \dots$, be a sequence
of compact geodesically complete $\CAT(\kappa)$ spaces.
Assume that $X_i$ converges to some $X$
in the Gromov-Hausdorff topology,
then
\[
\lim_{i \to \infty} \dim X_i = \dim X.
\]
\end{lem}

\begin{proof}
By Lemma \ref{lem: lstab},
it is enough to 
show the upper semi-continuity $\limsup_{i \to \infty} \dim X_i \le \dim X$.
We may assume that $\dim X$ is finite. 
Set $n = \dim X$.
Then all the spaces of directions in $X$ have dimension $\le n-1$.
Suppose that
the sequence $X_i$, $i = 1, 2, \dots$, 
has a subsequence $X_j$, $j = 1, 2, \dots$, such that
$\dim X_j \ge n+1$ for all $j$.
Then we can take a sequence $x_j \in X_j$, $j = 1, 2, \dots$, 
such that $\dim \Sigma_{x_j}X_j \ge n$ for all $j$,
and a point $x \in X$ 
to which the sequence $x_j \in X_j$, $j = 1, 2, \dots$, converges.
Since $\dim \Sigma_xX \le n-1$,
we have $\dim \Sigma_{x_j}X_j \le n-1$ for all sufficiently large $j$
(\cite[Lemma 11.5]{lytchak-nagano1}).
This is a contradiction, and proves the upper semi-continuity.
\end{proof}

On the pureness on the dimension, 
we have:

\begin{lem}\label{lem: npure}
Let $(X_i,p_i)$, $i = 1, 2, \dots$, be a sequence
of pointed proper geodesically complete $\CAT(\kappa)$ spaces.
Assume that
$(X_i,p_i)$ converges to some $(X,p)$
in the pointed Gromov-Hausdorff topology.
If each $X_i$ is purely $n$-dimensional,
then so is $X$.
\end{lem}

\begin{proof}
Assume that each $X_i$ is purely $n$-dimensional.
Then all the spaces of directions in $X_i$ have dimension $n-1$.
From Lemma \ref{lem: lstab} we derive $\dim X \le n$,
so that all the spaces of directions in $X$ have dimension $\le n-1$.
Moreover, we see $\dim X = n$.
Indeed,
if we would have $\dim X \le n-1$, then we could 
find a point $x_0 \in X$ with $\dim \Sigma_{x_0}X \le n-2$,
and a sequence $x_i \in X_i$, $i = 1, 2, \dots$, 
converging to the point $x_0 \in X$,
so that $\dim \Sigma_{x_i}X_i \le n-2$ for all sufficiently large $i$
(\cite[Lemma 11.5]{lytchak-nagano1}).
Similarly, we see that for every $x \in X$ we have $\dim \Sigma_xX = n-1$.
Therefore $X$ is purely $n$-dimensional too.
\end{proof}

We say that a separable metric space is
\emph{pure-dimensional}
if it is purely $n$-dimensional for some $n$.

We have the following characterization (\cite[Proposition 8.1]{lytchak-nagano2}):

\begin{prop}\label{prop: pure}
\emph{(\cite{lytchak-nagano2})}
Let $X$ be a proper, geodesically complete,
geodesic $\CAT(\kappa)$ space.
Let $W$ be a connected open subset of $X$.
Then the following are equivalent:
\begin{enumerate}
\item
$W$ is pure-dimensional;
\item
for every $p \in W$ the space $\Sigma_pX$ is pure-dimensional;
\item
for every $p \in W$ the space $T_pX$ is pure-dimensional.
\end{enumerate}
\end{prop}

%%%%%%%%%%
%%%%%%%%%%

%%%%%%%%%%
%Section 3: Topological regularity on volume growths
%%%%%%%%%%
\section{Topological regularity on volume growths}

In this section,
we discuss direct consequences of the study in \cite{nagano1}
and the studies in \cite{lytchak-nagano1} and \cite{lytchak-nagano2}.

%%%%%%%%%%%
\subsection{Volume comparisons of CAT$\boldsymbol{(\kappa)}$ spaces}

We recall that for every proper, geodesically complete $\CAT(\kappa)$ space
$X$ of $\dim X = n$,
the support of $\Haus^n$ coincides with the set of all points $x \in X$
with $\dim \Sigma_xX = n - 1$
(\cite[Theorem 1.2]{lytchak-nagano1}).
We can reformulate the volume comparisons studied in \cite{nagano1}
in the following way.

Let $X$ be a proper, geodesically complete $\CAT(\kappa)$ space.
Let $p \in X$ be a point with $\dim \Sigma_pX = n-1$.
Then there exists $u \in \Sigma_pX$
such that
$\Sph^{n-2}$ is isometrically embedded into $\Sigma_u\Sigma_pX$
(\cite[Theorem B]{kleiner}, \cite[Theorem 1.3]{lytchak-nagano1}).
Since $\Sigma_pX$ is geodesically complete,
there exists a surjective $1$-Lipschitz map $\varphi_p$ from $\Sigma_pX$ onto 
the unit tangent sphere $\Sigma_oM_{\kappa}^n$ 
at a point $o \in M_{\kappa}^n$
with
$d(\varphi_p(u),\varphi_p(v)) = d(u,v)$ 
(\cite[Lemma 3.1]{nagano1}, \cite[Lemma 2.2]{lytchak}, 
\cite[Proposition 11.3]{lytchak-nagano1}).
For every $r \in (0,D_{\kappa})$,
there exists a surjective $1$-Lipschitz map 
$\Phi_p \colon B_r(p) \to B_r(o)$
defined by
$\Phi_p(x) := \exp_o d(p,x) \varphi_p(x_p')$,
where $\exp_o$ is the exponential map at $o$.
The map $\Phi_p$ gives us an absolute volume comparison
of Bishop-G\"{u}nther type.
If in addition $X$ is purely $n$-dimensional,
then we see a volume rigidity (\cite[Proposition 6.1]{nagano1}).
Namely, we have:

\begin{prop}\label{prop: abvolcomp}
\emph{(\cite{nagano1})}
Let $X$ be a proper, geodesically complete $\CAT(\kappa)$ space,
and let $p \in X$ be a point with $\dim \Sigma_pX = n-1$.
Then for every $r \in (0,D_{\kappa})$ we have
\[
\Haus^n \bigl( B_r(p) \bigr) \ge \omega_{\kappa}^n(r).
\]
Moreover,
if in addition $X$ is purely $n$-dimensional,
then the equality holds if and only if
the pair $(B_r(p),p)$ is isometric to $(B_r(o),o)$
for any point $o \in M_{\kappa}^n$.
\end{prop}

Furthermore, we have the following relative volume comparison
of Bishop-Gromov type (\cite[Proposition 6.3]{nagano1}):

\begin{prop}\label{prop: relvolcomp}
\emph{(\cite{nagano1})}
Let $X$ be a proper, geodesically complete $\CAT(\kappa)$ space,
and let $p \in X$ be a point with $\dim \Sigma_pX = n-1$.
Then the function $f \colon (0,D_{\kappa}) \to [1,\infty]$ defined as
\[
f(t) := \frac{\Haus^n \bigl( B_t(p) \bigr)}{\omega_{\kappa}^n(t)}
\]
is monotone non-decreasing.
\end{prop}

%%%%%%%%%%%
\subsection{Volume convergence of CAT$\boldsymbol{(\kappa)}$ spaces}

Let $X_i$, $i = 1, 2, \dots$,
be a sequence of compact geodesically complete 
$\CAT(\kappa)$ spaces of $\dim X_i = n$.
Assume that
$X_i$ converges to some compact metric space $X$
in the Gromov-Hausdorff topology.
By Lemmas \ref{lem: stab} and \ref{lem: npure}, 
the compact, geodesically complete $\CAT(\kappa)$ space 
$X$ satisfies $\dim X = n$;
if in addition each $X_i$ is purely $n$-dimensional,
then so is $X$.

We can quote the volume convergence theorem for $\CAT(\kappa)$ spaces
in \cite[Theorem 1.1]{nagano1} in the following form:

\begin{thm}\label{thm: volconv}
\emph{(\cite{nagano1})}
Let $X_i$, $i = 1, 2, \dots$, 
be a sequence of compact, geodesically complete 
$\CAT(\kappa)$ spaces of $\dim X_i = n$.
If $X_i$ converges to some compact metric space $X$
in the Gromov-Hausdorff topology,
then 
\[
\Haus^n(X) =
\lim_{i \to \infty} \Haus^n(X_i).
\]
\end{thm}

From Proposition \ref{prop: abvolcomp}
we deduce the following
(\cite[Proposition 6.5]{nagano2}):

\begin{prop}\label{prop: precpt}
\emph{(\cite{nagano2})}
Let $c \in (0,\infty)$.
Then every isometry class of purely $n$-dimensional compact 
geodesically complete $\CAT(\kappa)$ spaces 
whose $n$-dimensional Hausdorff measures are bounded above by $c$
are precompact
in the Gromov-Hausdorff topology.
\end{prop}

We have the following infinitesimal regularity of Hausdorff measures
on $\CAT(\kappa)$ spaces 
(\cite[Theorem 1.4]{nagano1}):

\begin{thm}\label{thm: infvolconv}
\emph{(\cite{nagano1})}
Let $X$ be a proper, geodesically complete $\CAT(\kappa)$ space,
and let $p \in X$ be a point with $\dim \Sigma_pX = n-1$.
Then
\[
\lim_{t \to 0} \frac{\Haus^n \bigl( B_t(p) \bigr)}{t^n}
= \Haus^n \bigl( B_1(o_p) \bigr),
\]
where $B_1(o_p)$ is the unit ball centered at the vertex $o_p$ in $T_pX$.
\end{thm}

\begin{rem}\label{rem: ainfvolconv}
Lytchak and the author have generalized
Theorems \ref{thm: volconv}, \ref{thm: infvolconv},
and Proposition \ref{prop: precpt} 
for a canonical geometric volume measure
in \cite[Theorems 1.4, 1.5, and Subsection 12.5]{lytchak-nagano1}.
\end{rem}

%%%%%%%%%%%
\subsection{Topological regularity}

In what follows, we will use:

\begin{lem}\label{lem: volgrowth}
Let $X$ be a proper, geodesically complete $\CAT(\kappa)$ space,
and let $p \in X$ be a point with $\dim \Sigma_pX = n-1$.
Then for every $r \in (0,D_{\kappa})$ we have
\[
\frac{\Haus^{n-1} (\Sigma_pX)}{\Haus^{n-1} (\Sph^{n-1})} \le
\frac{\Haus^n \bigl( B_r(p) \bigr)}{\omega_{\kappa}^n(r)}.
\]
\end{lem}

\begin{proof}
From Theorem \ref{thm: infvolconv} and Proposition \ref{prop: relvolcomp},
we derive
\[
\frac{\Haus^{n-1} (\Sigma_pX)}{\Haus^{n-1} (\Sph^{n-1})}
= \frac{\Haus^n \bigl( B_1(o_p) \bigr)}{\omega_0^n(1)}
= \lim_{t \to 0} \frac{\Haus^n \bigl( B_t(p) \bigr)}{\omega_{\kappa}^n(t)}
\le \frac{\Haus^n \bigl( B_r(p) \bigr)}{\omega_{\kappa}^n(r)}
\]
(cf.~\cite[Remark 2.10]{nagano2}),
as required.
\end{proof}

Now we prove the following regularity:

\begin{thm}\label{thm: vgtopreg}
Let $X$ be a proper, geodesically complete $\CAT(\kappa)$ space,
and let $W$ be a purely $n$-dimensional open subset of $X$.
If for every $x \in W$ there exists $r \in (0,D_{\kappa})$ satisfying
\begin{equation}
\frac{\Haus^n \bigl( B_r(x) \bigr)}{\omega_{\kappa}^n(r)} < \frac{3}{2},
\label{eqn: vgtoprega}
\end{equation}
then $W$ is a topological $n$-manifold.
\end{thm}

\begin{proof}
By Lemma \ref{lem: volgrowth},
for every $x \in W$ the condition \eqref{eqn: vgtoprega} leads to
$\Haus^{n-1} (\Sigma_xX) < (3/2) \Haus^{n-1} (\Sph^{n-1})$.
From Proposition \ref{prop: pure} it follows that
$\Sigma_xX$ is a purely $(n-1)$-dimensional, compact, 
geodesically complete $\CAT(1)$ space.
Due to the volume sphere theorem \cite[Theorem 8.3]{lytchak-nagano2},
the space $\Sigma_xX$ is homeomorphic to $\Sph^{n-1}$.
Applying the local topological regularity theorem 
\cite[Theorem 1.1]{lytchak-nagano2} to $W$,
we conclude that $W$ is a topological $n$-manifold.
\end{proof}

\begin{rem}\label{rem: aftervgtopreg}
The assumption \eqref{eqn: vgtoprega} in Theorem \ref{thm: vgtopreg}
is optimal.
\end{rem}

\begin{exmp}\label{exmp: rescaledsusp}
For $\kappa \in (0,\infty)$,
let $X$ be the $(1/\sqrt{\kappa})$-rescaled space
$(1/\sqrt{\kappa}) (\Sph^{n-1} \ast T)$
of the spherical join $\Sph^{n-1} \ast T$.
The space $X$ is a purely $n$-dimensional, compact, 
geodesically complete $\CAT(\kappa)$ space,
and not a topological $n$-manifold
at any point in the spherical factor $\Sph^{n-1}$.
For every point $x \in X$ in the spherical factor $\Sph^{n-1}$ of $X$,
and for every $r \in (0,D_{\kappa})$,
we have $\Haus^n \bigl( B_r(x) \bigr) / \omega_{\kappa}^n(r) = 3/2$.
\end{exmp}

\begin{exmp}\label{exmp: kcone}
For $\kappa \in (-\infty,0]$,
let $X$ be the $\kappa$-cone $C_{\kappa}(\Sph^{n-2} \ast T)$ 
over $\Sph^{n-2} \ast T$
(see \cite[Definition I.5.6]{bridson-haefliger}).
Since $\Sph^{n-2} \ast T$ is $\CAT(1)$,
the $\kappa$-cone $X$ is a purely $n$-dimensional, proper, 
geodesically complete $\CAT(\kappa)$ space,
and not a topological $n$-manifold
at the vertex $o$ of the cone $X$.
For every $r \in (0,D_{\kappa})$,
we have $\Haus^n \bigl( B_r(o) \bigr) / \omega_{\kappa}^n(r) = 3/2$.
\end{exmp}

%%%%%%%%%%
%%%%%%%%%%

%%%%%%%%%%
%Section 4: A classification of CAT(1) spaces of small volume
%%%%%%%%%%
\section{A classification of CAT(1) spaces of small volume}

%%%%%%%%%%
\subsection{Spherical convex subsets}

We begin with the following:

\begin{prop}\label{prop: building}
Let $X$ be a compact, geodesically complete $\CAT(1)$ space
of $\dim X = n$
with decomposition $X = \bigcup_{i=1}^3 \Sigma_i$
for some subsets $\Sigma_1, \Sigma_2, \Sigma_3$ of $X$
satisfying the following:
\begin{enumerate}
\item
$\Sigma_i$ is a closed $\pi$-convex subset in $X$
that is isometric to $\Sph^n$
for all $i \in \{ 1, 2, 3 \}$;
\item
$\Sigma_i \subset \Sigma_j \cup \Sigma_k$
for all $i, j, k \in \{ 1, 2, 3 \}$.
\end{enumerate}
Then $X$ is isometric to either $\Sph^{n-1} \ast T$ or $\Sph^n$.
\end{prop}

\begin{proof}
If for some distinct $i, j \in \{ 1, 2, 3 \}$
we have $\Sigma_i \subset \Sigma_j$,
then by (1) we have $\Sigma_i = \Sigma_j$,
and hence by (2) we see that
$X$ is isometric to $\Sph^n$.

Assume now that
for all distinct $i, j \in \{ 1, 2, 3 \}$
we have $\Sigma_i \not\subset \Sigma_j$.
Put $\Sigma_{ij} := \Sigma_i \cap \Sigma_j$.
Then $\Sigma_{ij}$
is isometric to a non-empty, proper, closed $\pi$-convex subset of $\Sph^n$,
and hence contained in a closed unit $n$-hemisphere.
Therefore
we have $\Haus^n(\Sigma_{ij}) \le (1/2) \Haus^n(\Sph^n)$.
From (2) we derive $X = \Sigma_{12} \cup \Sigma_{23} \cup \Sigma_{31}$.
Hence we obtain
\[
\Haus^n(X) \le 
\Haus^n(\Sigma_{12}) + \Haus^n(\Sigma_{23}) + \Haus^n(\Sigma_{31})
\le \frac{3}{2} \Haus^n(\Sph^n).
\]
From the present assumption it follows that 
$X$ is not a topological $n$-manifold;
indeed,
for all points in $\Sigma_i - \Sigma_j$,
the closest points on $\Sigma_{ij}$ possess no neighborhoods
homeomorphic to $\R^n$.
Since $X$ is not homeomorphic to $\Sph^n$,
the volume sphere theorem \cite[Theorem 8.3]{lytchak-nagano2}
implies $\Haus^n(X) \ge (3/2) \Haus^n(\Sph^n)$.
Thus for all distinct $i, j$ we have
$\Haus^n(\Sigma_{ij}) = (1/2) \Haus^n(\Sph^n)$;
in particular,
$\Sigma_{ij}$ is isometric to a closed unit $n$-hemisphere.
Since by (2) we have $X = \Sigma_i \cup \Sigma_j$ for all distinct $i, j$,
we conclude that
$X$ is isometric to $\Sph^{n-1} \ast T$.
\end{proof}

\begin{rem}\label{rem: referee}
An anonymous referee tells us the proof of 
Proposition \ref{prop: building} discussed above.
In a previous manuscript of this paper,
when we proved Proposition \ref{prop: building},
we used 
the lune lemma of Ballmann and Brin \cite[Lemma 2.5]{ballmann-brin},
the metric characterizations of spherical buildings 
of Balser and Lytchak \cite[Theorems 1.1 and 1.4]{balser-lytchak},
and the spherical join decomposition theorem of Lytchak \cite[Corollary 1.2]{lytchak}.
\end{rem}

%%%%%%%%%%%%%%%%%%%%%%%%%%%%%%%%%%%%%%%%%
\subsection{Proof of Theorem \ref{thm: critical}}

If $X$ is a purely $1$-dimensional, compact, 
geodesically complete $\CAT(1)$ space 
with $\Haus^1(X) = (3/2) \Haus^1(\Sph^1)$,
and if $X$ has a tripod, then
$X$ is isometric to either the $1$-triplex or $\Sph^0 \ast T$.
For the $2$-dimensional case,
we know the following (\cite[Theorem B]{nagano2}):

\begin{prop}\label{prop: extremal2}
\emph{(\cite{nagano2})}
Let $X$ be a purely $2$-dimensional, compact, 
geodesically complete $\CAT(1)$ space
with $\Haus^2(X) = (3/2) \Haus^2(\Sph^2)$.
If $X$ has a tripod, then
$X$ is isometric to either the $2$-triplex or $\Sph^1 \ast T$.
\end{prop}

To finish the proof of Theorem \ref{thm: critical},
we show:

\begin{prop}\label{prop: extremal}
Let $X$ be a purely $n$-dimensional, compact, 
geodesically complete $\CAT(1)$ space
satisfying \eqref{eqn: criticala}.
If $X$ has a tripod, then
$X$ is isometric to either the $n$-triplex or $\Sph^{n-1} \ast T$.
\end{prop}

\begin{proof}
By Proposition \ref{prop: extremal2}, we may assume $n \ge 3$.
Let $p_1, p_2, p_3 \in X$ be elements of a tripod.
By Proposition \ref{prop: abvolcomp} and \eqref{eqn: criticala},
we see that $X$ has the decomposition
$X = \bigcup_{i=1}^3 B_{\pi/2}(p_i)$
such that
$B_{\pi/2}(p_i)$ is isometric to a closed unit $n$-hemisphere for each $i \in \{ 1, 2, 3 \}$;
in particular,
$\partial B_{\pi/2}(p_i)$ is isometric to $\Sph^{n-1}$.
Put $\Sigma_i := \partial B_{\pi/2}(p_i)$.

Let $Y := \bigcup_{i=1}^3 \Sigma_i$.
Note that $\Sigma_i$ is a closed $\pi$-convex subset in $X$
for each $i \in \{ 1, 2, 3 \}$.
The geodesical completeness of $X$ implies that
$\Sigma_i$ is contained in $\Sigma_j \cup \Sigma_k$
for all $i, j, k \in \{ 1, 2, 3 \}$.

We show that $\Sigma_i \cup \Sigma_j$ is $\pi$-convex in $X$
for all distinct $i, j \in \{ 1, 2, 3 \}$.
For distinct $i, j \in \{ 1, 2, 3 \}$,
take $y_1, y_2$ in $\Sigma_i \cup \Sigma_j$
with $d(y_1,y_2) < \pi$,
and let $y_1y_2$ be the geodesic joining them.
We may assume that
$y_1 \in \Sigma_i - \Sigma_j$ and $y_2 \in \Sigma_j - \Sigma_i$.
By the geodesical completeness of $X$,
the points $y_1$ and $y_2$ must belong to $\Sigma_k$
for $k \in \{ 1, 2, 3 \}$ distinct to $i, j$.
The $\pi$-convexity of $\Sigma_k$ in $X$
implies that $y_1y_2$ is contained in $\Sigma_k$,
and hence $y_1y_2$ is contained in $\Sigma_i \cap \Sigma_j$;
indeed, the set $\Sigma_k$ is contained in $\Sigma_i \cup \Sigma_j$.
Hence $\Sigma_i \cup \Sigma_j$ is $\pi$-convex.

Since $\Sigma_i \cup \Sigma_j$ is $\pi$-convex in $X$
for all distinct $i, j \in \{ 1, 2, 3 \}$,
so is the whole $Y$.
Note that every closed $\pi$-convex subspace of a $\CAT(1)$ space 
is also $\CAT(1)$.
Hence the subspace $Y$ is a $\CAT(1)$ space of $\dim Y = n-1$.
The present assumption $n \ge 3$ implies $\dim Y \ge 2$.
Observe that $Y$ is geodesically complete,
and $\Sigma_i$ is also closed and $\pi$-convex in $Y$ 
for all $i \in \{ 1, 2, 3 \}$.
By Proposition \ref{prop: building},
we see that $Y$ is isometric to either $\Sph^{n-2} \ast T$ or $\Sph^{n-1}$.
If $Y$ is $\Sph^{n-2} \ast T$, then $X$ is isometric to the $n$-triplex.
If $Y$ is $\Sph^{n-1}$, then $X$ is isometric to $\Sph^{n-1} \ast T$.
\end{proof}

Proposition \ref{prop: extremal} 
and the capacity sphere theorem
\cite[Theorem 1.5]{lytchak-nagano2}
complete the proof of Theorem \ref{thm: critical}.
\qed

%%%%%%%%%%
%%%%%%%%%%

%%%%%%%%%%
%Section 5: A sphere theorem for CAT(1) homology manifolds 
%%%%%%%%%%
\section{A sphere theorem for CAT(1) homology manifolds}

%%%%%%%%%%
\subsection{Homology manifolds}

Let $H_{\ast}$ denote the singular homology with $\Z$-coefficients.
A locally compact, separable metric space $M$ is said to be a
\emph{homology $n$-manifold} 
if for every $p \in M$ 
the local homology
$H_{\ast}(M,M - \{ p \})$ at $p$
is isomorphic to $H_{\ast}(\R^n,\R^n - \{ 0 \})$,
where $0$ is the origin of $\R^n$.
A homology $n$-manifold $M$ is a 
\emph{generalized $n$-manifold} 
if $M$ is an $\ANR$ of $\dim M < \infty$.
Every generalized $n$-manifold has dimension $n$.
Due to the theorem of Moore (see \cite[Chapter IV]{wilder}),
for each $n \in \{ 1, 2 \}$,
every generalized $n$-manifold is a topological $n$-manifold.

Every homology $n$-manifold with an upper curvature bound
is a geodesically complete generalized $n$-manifold.
Thurston \cite[Theorem 3.3]{thurston} proved that
every homology $3$-manifold with an upper curvature bound
is a topological $3$-manifold.
We refer the readers to \cite{lytchak-nagano2} for advanced studies
of homology manifolds with an upper curvature bound.

We recall the following 
(\cite[Lemma 3.1 and Corollary 3.4]{lytchak-nagano2}):

\begin{prop}\label{prop: hm}
{\em (\cite{lytchak-nagano2})}
Let $X$ be a metric space with an upper curvature bound.
A locally compact open subset $M$ of $X$
is a homology $n$-manifold
if and only if 
for every $p \in M$ the space $\Sigma_pX$
has the same homology as $\Sph^{n-1}$;
in this case,
$\Sigma_pX$ is a homology $(n-1)$-manifold
and $T_pX$ is a homology $n$-manifold.
\end{prop}

\begin{rem}\label{rem: afterhm}
Lytchak and the author \cite[Theorems 1.2 and 6.5]{lytchak-nagano2} 
have proved that
for every homology $n$-manifold $M$ with an upper curvature bound
there exists a locally finite subset $E$ of $M$ such that
$M-E$ is a topological $n$-manifold;
moreover, every point in $M$ has a neighborhood
homeomorphic to some cone over a closed topological $(n-1)$-manifold.
\end{rem}

%%%%%%%%%%
\subsection{Locally geometrical contractivity}

Following the terminology in \cite{grove-petersen-wu} and \cite{petersen}, we say that 
a function $\rho \colon [0,r) \to [0,\infty)$ with $\rho(0) = 0$ is a
\emph{contractivity function}
if $\rho$ is continuous at $0$,
and if $\rho \ge \id_{[0,r)}$,
where $\id_{[0,r)}$ is the identity function on $[0,r)$.
For a contractivity function $\rho \colon [0,r) \to [0,\infty)$,
a metric space $X$ is 
\emph{$\LGC(\rho)$},
\emph{locally geometrically contractible with respect to $\rho$},
if for every $p \in X$ and for every $s \in (0,r)$
the ball $B_s(p)$ is contractible inside 
the concentric ball $B_{\rho(s)}(p)$.
Every $\CAT(\kappa)$ space
is $\LGC(\id_{[0,D_{\kappa})})$.

We recall the following,
which is just a combination of
the theorem of Grove, Petersen, and Wu \cite[Theorem 2.1]{grove-petersen-wu} 
and the theorems of Petersen \cite[Theorem A, and Theorem in Section 5]{petersen}.

\begin{thm}\label{thm: gpw}
\emph{(\cite{grove-petersen-wu}, \cite{petersen})}
Let $\rho \colon [0,r) \to [0,\infty)$ be a contractivity function.
If a sequence of compact $\LGC(\rho)$ spaces $X_i$, $i = 1, 2, \dots$,  
of dimension $\le n$ converges to
some compact metric space $X$ of finite dimension
in the Gromov-Hausdorff topology,
then 
\begin{enumerate}
\item
$X$ is an $\LGC(\rho)$ space of $\dim X \le n$;
\item
$X$ is homotopy equivalent to $X_i$
for all sufficiently large $i$;  
\item
if in addition each $X_i$ is a topological $n$-manifold,
then $X$ is a generalized $n$-manifold.
\end{enumerate}
\end{thm}

For sequences of compact $\CAT(\kappa)$ homology $n$-manifolds,
we have the following
(\cite[Lemma 3.3]{lytchak-nagano2}):

\begin{lem}\label{lem: stabcathm}
\emph{(\cite{lytchak-nagano2})}
If a sequence of compact $\CAT(\kappa)$ homology $n$-manifolds 
$X_i$, $i = 1, 2, \dots$, converges to
some compact metric space $X$ in the Gromov-Hausdorff topology,
then $X$ is a homology $n$-manifold
\end{lem}

\begin{rem}\label{rem: agpw}
Lytchak and the author \cite[Theorems 1.3 and 7.5]{lytchak-nagano2}
have proved that
if a sequence of compact $\CAT(\kappa)$ Riemannian $n$-manifolds 
$X_i$, $i = 1, 2, \dots$, converges to
some compact metric space $X$ in the Gromov-Hausdorff topology,
then $X$ is a topological $n$-manifold in which
all iterated spaces of directions are homeomorphic to spheres;
moreover,
$X$ is homeomorphic to $X_i$ for all sufficiently large $i$.
\end{rem}

%%%%%%%%%%
\subsection{The Lusternik-Schnirelmann category}

Let $X$ be a topological space $X$.
The \emph{Lusternik-Schnirelmann category of $X$},
denoted by $\cat X$, is defined as the least non-negative integer $k$
such that there exists an open covering of $X$
consisting of $k+1$ contractible subsets in $X$
(possibly $\infty$ if such a finite covering does not exist).
By definition, it follows that
$\cat X = 0$ if and only if $X$ is contractible.
Notice that
$\cat X$ depends only on the homotopy type of $X$.

The following seems to be well-known for experts.

\begin{prop}\label{prop: hs}
If a connected,
compact topological $n$-manifold $M$ satisfies $\cat M = 1$,
then $M$ is homotopy equivalent to $\Sph^n$;
in particular, $M$ is homeomorphic to $\Sph^n$.
\end{prop}

The second half of Proposition \ref{prop: hs}
is derived from the first half and
the resolutions of the Poincar\'{e} conjecture due to Perelman,
and the generalized Poincar\'{e} conjecture due to Freedman and Smale.
The first half is well-known in algebraic topology
(see \cite[Lemma 8.2]{lytchak-nagano2}).

%%%%%%%%%%
\subsection{A sphere theorem for topological manifolds}

Before showing Theorem \ref{thm: volsphhm},
we prove a weaker one:

\begin{prop}\label{prop: volsphmfd}
For every positive integer $n$,
there exists a positive number $\delta \in (0,\infty)$ 
such that
if a compact $\CAT(1)$ topological $n$-manifold $X$ satisfies
\eqref{eqn: volsphhma} for $\delta$,
then $X$ is homeomorphic to $\Sph^n$.
\end{prop}

\begin{proof}
Suppose the contrary.
By virtue of the volume sphere theorem \cite[Theorem 8.3]{lytchak-nagano2},
we may suppose that
there exists a sequence of compact $\CAT(1)$ topological $n$-manifolds
$X_i$, $i = 1, 2, \dots$,
with $\lim_{i \to \infty} \Haus^n(X_i) = (3/2) \Haus^n(\Sph^n)$
such that each $X_i$ is not homeomorphic to $\Sph^n$.
Due to the capacity sphere theorem 
\cite[Theorem 1.5]{lytchak-nagano2},
we may assume that
each $X_i$ has a tripod.
By Proposition \ref{prop: precpt},
the sequence $X_i$, $i = 1, 2, \dots$,
has a convergent subsequence $X_j$,
$j = 1, 2, \dots$,
tending to some 
compact metric space $X$
in the Gromov-Hausdorff topology.
By Lemma \ref{lem: npure},
the limit
$X$ is a purely $n$-dimensional, compact, 
geodesically complete $\CAT(1)$ space.
Since each $X_j$ has a tripod,
so does $X$.
From Theorem \ref{thm: volconv} we derive
$\Haus^n(X) = (3/2) \Haus^n(\Sph^n)$.

Theorem \ref{thm: critical} implies that
$X$ is isometric to either the $n$-triplex
or $\Sph^{n-1} \ast T$.
Theorem \ref{thm: gpw} tells us that
$X$ is a homology $n$-manifold,
and hence it must be isometric to the $n$-triplex. 
Recall that the $n$-triplex $X$ consists of $\Sph^{n-2}\ast T$
and the three copies of unit $n$-hemispheres
(see Example \ref{exmp: triplex}).
Hence we find a pair of two points $p_1, p_2$ 
in the $\pi$-convex subset $\Sph^{n-2}$ of $X$
with $d(p_1,p_2) = \pi$
such that
$X = B_{\pi/2}(p_1) \cup B_{\pi/2}(p_2)$.
Take $p_{j,k} \in X_j$, $k = 1, 2$,
converging to $p_k \in X$ as $j \to \infty$.
Then $X_j = U_{3\pi/4}(p_{j,1}) \cup U_{3\pi/4}(p_{j,2})$,
provided $j$ is large enough;
in particular,
$\cat X_j = 1$.
From Proposition \ref{prop: hs} it follows that
$X_j$ is homeomorphic to $\Sph^n$.
Thus we obtain a contradiction.
This completes the proof.
\end{proof}

%%%%%%%%%%
\subsection{A homotopy sphere theorem for homology manifolds}

Similarly to Proposition \ref{prop: volsphmfd},
we obtain:

\begin{prop}\label{prop: volhomsphhm}
For every positive integer $n$,
there exists a positive number $\delta \in (0,\infty)$ 
such that
if a compact $\CAT(1)$ homology $n$-manifold $X$ satisfies
\eqref{eqn: volsphhma} for $\delta$,
then $X$ is homotopy equivalent to $\Sph^n$.
\end{prop}

\begin{proof}
Suppose the contrary.
Similarly to the proof of Proposition \ref{prop: volsphmfd},
we may suppose that there exists a sequence of compact $\CAT(1)$
homology $n$-manifolds
$X_i$, $i = 1, 2, \dots$,
with $\lim_{i \to \infty} \Haus^n(X_i) = (3/2) \Haus^n(\Sph^n)$
such that each $X_i$ 
is not homotopy equivalent to $\Sph^n$.
By Proposition \ref{prop: precpt},
the sequence $X_i$, $i = 1, 2, \dots$,
has a convergent subsequence $X_j$, $j = 1, 2, \dots$,
tending to some $X$ in the Gromov-Hausdorff topology.
By Lemma \ref{lem: npure},
the limit $X$ is a purely $n$-dimensional, compact, 
geodesically complete $\CAT(1)$ space.
From Theorem \ref{thm: volconv} we derive
$\Haus^n(X) = (3/2) \Haus^n(\Sph^n)$.
Due to Theorem \ref{thm: critical},
we see that $X$ is either homeomorphic to $\Sph^n$
or isometric to $\Sph^{n-1} \ast T$.
By Lemma \ref{lem: stabcathm},
the limit $X$ must be a homology $n$-manifold,
and hence homeomorphic to $\Sph^n$.
It follows from Theorem \ref{thm: gpw} that
$X_j$ is homotopy equivalent to an $n$-sphere $X$
for all sufficiently large $j$.
This is a contradiction.
\end{proof}

\begin{rem}\label{rem: aftervolhomsphhm}
Combining Proposition \ref{prop: volhomsphhm} and 
the resolutions of the Poincar\'{e} conjecture
also lead to Proposition \ref{prop: volsphmfd}.
\end{rem}

%%%%%%%%%%
\subsection{Proof of Theorem \ref{thm: vgtopreghm}}

Let $\delta \in (0,\infty)$ be sufficiently small.
Let $X$ be a $\CAT(\kappa)$ homology $n$-manifold,
and let $W$ be an open subset of $X$.
Assume that for every $x \in W$ there exists $r \in (0,D_{\kappa})$
satisfying \eqref{eqn: vgtopreghma}.
By Lemma \ref{lem: volgrowth},
for every $x \in W$ the condition \eqref{eqn: vgtopreghma} leads to
\[
\frac{\Haus^{n-1} (\Sigma_xX)}{\Haus^{n-1} (\Sph^{n-1})} < \frac{3}{2} + \delta.
\]
From Proposition \ref{prop: hm} it follows that
$\Sigma_xX$ is a compact $\CAT(1)$ homology $(n-1)$-manifold.
By Proposition \ref{prop: volhomsphhm}, the space
$\Sigma_xX$ is homotopy equivalent to $\Sph^{n-1}$,
provided $\delta$ is small enough.
Due to the local topological regularity theorem 
\cite[Theorem 1.1]{lytchak-nagano2},
we conclude that $W$ is a topological $n$-manifold.
Thus we obtain Theorem \ref{thm: vgtopreghm}.
\qed

%%%%%%%%%%
\subsection{Proof of Theorem \ref{thm: volsphhm}}

Let $\delta \in (0,\infty)$ be sufficiently small.
Let $X$ be a compact $\CAT(1)$ homology $n$-manifold
satisfying \eqref{eqn: volsphhma}.
From Propositon \ref{prop: relvolcomp} and \eqref{eqn: volsphhma},
for every $x \in X$ we derive
\[
\frac{\Haus^n \bigl( B_r(x) \bigr)}{\omega_1^n(r)} 
\le \frac{\Haus^n (X)}{\Haus^n (\Sph^n)} 
< \frac{3}{2} + \frac{\delta}{\Haus^n (\Sph^n)} 
\]
for any $r \in (0,\pi)$.
From Theorem \ref{thm: vgtopreghm} we deduce that
$X$ is a topological $n$-manifold, 
provided $\delta$ is small enough.
This together with Proposition \ref{prop: volsphmfd}
completes the proof of Theorem \ref{thm: volsphhm}.
\qed

%%%%%%%%%%
%%%%%%%%%%

%%%%%%%%%%
%Appendix A: Three-manifold recognition revisited 
%%%%%%%%%%
\appendix
\section{Three-manifold recognition revisited}

One of the key points in the proof of Theorem \ref{thm: volsphhm}
is to prove Proposition \ref{prop: volsphmfd}.
In the proof of Proposition \ref{prop: volsphmfd} discussed above,
we rely on the resolutions of the (generalized) Poincar\'{e} conjecture
when we use Proposition \ref{prop: hs}.
As explained below,
we can prove Proposition \ref{prop: volsphmfd} in the $3$-dimensional case
without relying on the Poincar\'{e} conjecture.

A locally compact, separable metric space $M$ is said to be a 
\emph{homology $n$-manifold with boundary}
if for every $p \in M$ there exists $x \in D^n$ 
such that $H_{\ast}(M,M-\{p\})$ coincides with $H_{\ast}(D^n,D^n-\{x\})$,
where $D^n$ is the Euclidean closed unit $n$-disk
centered at the origin in $\R^n$;
the 
\emph{boundary $\partial M$ of $M$} 
is defined as the set of all points $p \in M$
at which the local homologies $H_{\ast}(M,M-\{p\})$ are trivial.
A homology $n$-manifold $M$ with boundary is a 
\emph{generalized $n$-manifold with boundary}
if $M$ is an $\ANR$ of $\dim M < \infty$.
If $M$ is a generalized $n$-manifold with boundary,
then $\dim M = n$, 
and $\partial M$ is closed and nowhere dense in $M$
(see e.g., \cite[Lemma 2]{mitchell}).
From the theorem of Mitchell \cite[Theorem]{mitchell},
it follows that if $M$ is a generalized $n$-manifold with boundary,
then $\partial M$ is either empty or
a generalized $(n-1)$-manifold without boundary.

Thurston showed in \cite[Proposition 2.7]{thurston} that
if $X$ is a $\CAT(\kappa)$ homology $n$-manifold,
then for every $r \in (0,D_{\kappa}/2)$, and for every $p \in X$,
the compact contractible metric ball $B_r(x)$ is a generalized $n$-manifold
with boundary $\partial B_r(p)$;
in particular,
by the theorem of Mitchel \cite[Theorem]{mitchell},
and the Poincar\'{e} duality for homology manifolds (see e.g., \cite{bredon}), 
the metric sphere $\partial B_r(p)$ is a generalized $(n-1)$-manifold
with the same homology as $\Sph^{n-1}$
(see also \cite[Lemma 3.2]{lytchak-nagano2}).
This property holds true for any $r \in (0,D_{\kappa})$ beyond $D_{\kappa}/2$.

\begin{lem}\label{lem: bms}
{\em (cf.~\cite{thurston})}
If $X$ is a $\CAT(\kappa)$ homology $n$-manifold,
then for every $r \in (0,D_{\kappa})$, 
and for every $p \in X$,
the ball $B_r(p)$ is a generalized $n$-manifold with boundary $\partial B_r(p)$;
in particular,
$\partial B_r(p)$ is a generalized $(n-1)$-manifold
with the same homology as $\Sph^{n-1}$.
\end{lem}

\begin{proof}
For every $x \in \partial B_r(p)$,
the set $B_r(p)-\{x\}$ is contractible to $p$ inside itself
along the geodesics from $p$,
and hence the reduced homology
$\tilde{H}_{\ast}(B_r(p),B_r(p)-\{x\})$
is trivial
since we have the exact sequence
\begin{multline*}
\tilde{H}_k(B_r(p))
\to
\tilde{H}_k(B_r(p),B_r(p)-\{x\}) \\
\to
\tilde{H}_{k-1}(B_r(p)-\{x\})
\to
\tilde{H}_{k-1}(B_r(p))
\end{multline*}
for all $k \in \N$.
Hence $\partial B_r(p)$ is the boundary of $B_r(p)$
as generalized manifolds.
The theorem of Mitchel \cite[Theorem]{mitchell} together with
the Poincar\'{e} duality
leads to the second half of the lemma.
\end{proof}

From now on,
we focus on the $3$-dimensional case.
Thurston proved in \cite[Theorem 3.3]{thurston} that
if $p$ is a point in a $\CAT(\kappa)$ homology $3$-manifold,
then $U_r(p)$ is homeomorphic to $\R^3$ for any $r \in (0,D_{\kappa}/2)$
whose upper bound $D_{\kappa}/2$ guarantees the strong convexity of $U_r(p)$.
By the same arguments as in \cite{thurston},
we can prove the following:

\begin{thm}\label{thm: thurston}
{\em (cf.~\cite{thurston})}
Let $X$ be a $\CAT(\kappa)$ homology $3$-manifold.
Then for every $p \in X$, 
and for every $r \in (0,D_{\kappa})$,
the ball $U_r(p)$ is homeomorphic to $\R^3$.
\end{thm}

Reviewing the arguments discussed in \cite{thurston},
we sketch the proof.

Let $Y$ and $Z$ be topological spaces.
A map $f \colon Y \to Z$ is said to be
\emph{approximable by homeomorphisms},
abbreviated as $\ABH$,
if for every open covering $\mathcal{U}$ of $Z$
there exists a homeomorphism $h \colon Y \to Z$
such that for each $y \in Y$ we find $U \in \mathcal{U}$
with $f(y) \in U$ and $h(y) \in U$.
By the Daverman-Preston sliced shrinking theorem \cite[Theorem]{daverman-preston},
we already know that
if $f \colon Y \times \R \to Z$ is a proper, surjective continuous map
such that each fiber $f^{-1}(z)$ is contained in some slice $Y \times \{ t \}$,
and if each of the level maps of $f$ is $\ABH$,
then $f$ is $\ABH$ too.

Let $\tilde{H}^{\ast}$ denote
the reduced \v{C}ech cohomology with $\Z$-coefficients.
A proper surjective map $c \colon Y \to Z$
between locally compact Hausdorff spaces is said to be 
\emph{acyclic}
if $\tilde{H}^{\ast}(c^{-1}(z))$ is trivial for all $z \in Z$.

\begin{proof}[Proof of Theorem \ref{thm: thurston}]
Let $X$ be a $\CAT(\kappa)$ homology $3$-manifold.
Let $p \in X$ be arbitrary.
From Lemma \ref{lem: bms}
we see that for each $t \in (0,D_{\kappa})$
the metric sphere $\partial B_t(p)$ is homeomorphic to $\Sph^2$
since every generalized $2$-manifold is a topological $2$-manifold.

Take arbitrary $s, r \in (0,D_{\kappa})$ with $s < r$.
Let $c_{r,s} \colon \partial B_r(p) \to \partial B_s(p)$
be the continuous surjective map
defined as $c_{r,s}(x) := \gamma_{px}(s)$,
where $\gamma_{px} \colon [0,d(p,x)] \to X$ is the geodesic from $p$ to $x$.
Choose a point $z \in \partial B_s(p)$.
Let $\Gamma_z c_{r,s}^{-1}(z)$ be the geodesic cone in $X$
defined as 
\[
\Gamma_z c_{r,s}^{-1}(z) := \bigcup \{ \, zy \mid y \in c_{r,s}^{-1}(z) \, \}.
\]
By definition,
any point in $\Gamma_z c_{r,s}^{-1}(z)$ lies on a geodesic
joining $p$ and some point in $\partial B_r(p)$.
Hence $\Gamma_z c_{r,s}^{-1}(z)$ is contained in $B_r(p)$.

Observe that
$B_r(p) - \Gamma_{z}c_{r,s}^{-1}(z)$
is contractible to $p$ inside itself along the geodesics from $p$.
Following the same way as in \cite[Corollary 2.10]{thurston},
by Lemma \ref{lem: bms}
and the Alexander-Lefschetz duality for homology manifolds
(see e.g., \cite[Proposition 2.8]{thurston}),
we see that
$\tilde{H}^{\ast}(c_{r,s}^{-1}(z))$ is trivial.
Since $z$ is arbitrary in $\partial B_s(p)$,
the map $c_{r,s}$ is acyclic;
in particular,
each of the fibers of $c_{r,s}$ fails to separate the $2$-sphere $\partial B_r(p)$.
This implies that the map
$c_{r,s}$ is $\ABH$ for any $s, r \in (0,D_{\kappa})$ with $s < r$.

For a fixed $r \in (0,D_{\kappa})$,
we consider the proper, continuous surjective map
$f_r \colon \partial B_r(p) \times (0,r) \to U_r(p) - \{ p \}$
defined by $f_r(y,s) := c_{r,s}(y)$.
Each fiber $f_r^{-1}(z)$ is contained in some slice 
$\partial B_r(p) \times \{ s \}$.
Due to the Daverman-Preston sliced shrinking theorem \cite[Theorem]{daverman-preston},
we obtain a homeomorphism
between $\partial B_r(p) \times (0,r)$ and $U_r(p) - \{ p \}$,
and hence between $\R^3$ and $U_r(p)$.
Thus we conclude Theorem \ref{thm: thurston}.
\end{proof}

We give another proof of Proposition \ref{prop: volsphmfd}
in the $3$-dimensional case 
without using the resolution of the Poincar\'{e} conjecture.

\begin{proof}[Proof of Proposition \ref{prop: volsphmfd} in the $3$-dimensional case]
Suppose now that
there exists a sequence of compact $\CAT(1)$ topological $3$-manifolds
$X_i$, $i = 1, 2, \dots$,
with $\lim_{i \to \infty} \Haus^3(X_i) = (3/2) \Haus^3(\Sph^3)$
such that each $X_i$ is not homeomorphic to $\Sph^3$.
Similarly to the proof in Subsection 5.4,
we see that the sequence $X_i$, $i = 1, 2, \dots$,
has a convergent subsequence $X_j$,
$j = 1, 2, \dots$,
tending to the $3$-triplex $X$
in the Gromov-Hausdorff topology;
moreover, 
we find a pair of two points $p_1, p_2 \in X$ with $d(p_1,p_2) = \pi$
such that
$X = B_{\pi/2}(p_1) \cup B_{\pi/2}(p_2)$.
Take $p_{j,k} \in X_j$, $k = 1, 2$,
converging to $p_k \in X$ as $j \to \infty$.
Then $X_j = U_{3\pi/4}(p_{j,1}) \cup U_{3\pi/4}(p_{j,2})$,
provided $j$ is large enough.
From Theorem \ref{thm: thurston}
it follows that the balls 
$U_{3\pi/4}(p_{j,1})$ and $U_{3\pi/4}(p_{j,2})$ are homeomorphic to $\R^3$.
The generalized Schoenflies theorem (see e.g., \cite[Theorem 1.8.2]{rushing})
implies that $X_j$ is homeomorphic to $\Sph^3$,
and leads to a contradiction.
This completes the proof.
\end{proof}

%%%%%%%%%%
%%%%%%%%%%

%%%%%%%%%%
%bibliography
%%%%%%%%%%

%%%%%%%%%%
%%%%%%%%%%

\end{document}